\documentclass[a4paper,12pt]{article}
\usepackage[top=1.25in,left=1in,right=1in,bottom=1.25in]{geometry}           
\usepackage{latexsym}
\usepackage{epsfig, ecltree, epic, eepic}
\usepackage{enumerate,amsmath,amssymb,dsfont,pifont,amsthm}
\usepackage{graphicx}
\usepackage[arrow, matrix, curve]{xy}
\usepackage{pstricks}
\usepackage{pst-grad} 
\usepackage{pst-plot} 

\numberwithin{equation}{section}

\theoremstyle{plain}
\newtheorem{theorem}{Theorem}[section]
\newtheorem{corollary}[theorem]{Corollary}
\newtheorem{lemma}[theorem]{Lemma}
\newtheorem{proposition}[theorem]{Proposition}

\theoremstyle{definition}
\newtheorem{definition}[theorem]{Definition}

\theoremstyle{remark}
\newtheorem{remark}[theorem]{Remark}

\newtheorem{example}[theorem]{Example}

\newcommand{\bbr}{\mathbb{R}}
\newcommand{\bbq}{\mathbb{Q}}
\newcommand{\bbp}{\mathbb{P}}
\newcommand{\bbe}{\mathbb{E}}
\newcommand{\bbf}{\mathbb{F}}
\newcommand{\bbn}{\mathbb{N}}

\newcommand{\cb}{\mathcal{B}}
\newcommand{\cc}{\mathcal{C}}
\newcommand{\cf}{\mathcal{F}}

\newcommand{\pf}{\longrightarrow}

\newcommand{\abs}[1]{\left| #1 \right|}
\newcommand{\norm}[1]{\left\| #1 \right\|}

\newcommand{\gdw}{\Leftrightarrow}

\DeclareMathOperator{\sign}{sign}

\begin{document}

\allowdisplaybreaks

\title{\bfseries On Deterministic Markov Processes: Expandability and Related Topics}

\author{%
    \textsc{Alexander Schnurr}%
    \thanks{Lehrstuhl IV, Fakult\"at f\"ur Mathematik, Technische Universit\"at Dortmund,
              D-44227 Dortmund, Germany,
              \texttt{alexander.schnurr@math.tu-dortmund.de}, 
              phone: +49-231-755-3099, fax: +49-231-755-3064.}
    }

\date{\today}

\maketitle
\begin{abstract}
We analyze the class of universal Markov processes on $\bbr^d$ which do not depend on random. For this, as well as for several subclasses, we prove criteria whether a function $f:[0,\infty[\to\bbr^d$ can be a path of a process in the respective class. This is useful in particular in the construction of (counter-)examples. The semimartingale property is characterized in terms of the jumps of a one-dimensional deterministic Markov process. We emphasize the differences between the time homogeneous and the time inhomogeneous case and we show that a deterministic Markov process is in general more complicated than a Hunt process plus `jump structure'.  
\end{abstract}

\emph{MSC 2010:} 60J25 (primary), 26A30, 26A45, 60J35, 60J75, 47G30 (secondary)

\emph{Keywords:} Markov semimartingale, deterministic process, It\^o process, Feller process, symbol, expandability

\section{Introduction}
Deterministic processes arise naturally in several parts of the theory of stochastic processes. Examples include the space dependent drift being one part of a Feller process or the seasonal component of a time series in continuous time. Furthermore it is known that every PII (i.e. a c\`adl\`ag process with independent increments) can be written as the sum of a PII semimartingale and a deterministic process (cf. \cite{jacodshir} Theorem II.5.1).

When starting our investigation of deterministic universal Markov processes we had the following three questions in mind:\\
I) Is there a simple example of a Hunt semimartingale which is not an It\^o process? \\
II) Can we characterize every (one-dim.) deterministic Markov process? \\
III) Given a function $f:[0,\infty[ \to \bbr^d$. Can we directly say whether there exists a deterministic process of a certain class having this function as a path?  

In our recent paper \cite{detmp1} we have solved the first question by introducing the Cantor process (cf. Example \ref{ex:cantorprocess}). Furthermore we gave an affirmative answer to the second question in the special case of Hunt processes. In the present paper we treat the general case of deterministic Markov processes. Several results which are true for Hunt processes turn out to be wrong in this more general framework. For $d=1$ paths do not have to be monotone any more or even of finite variation. This leads to new questions like: when is a deterministic Markov process a semimartingale? The criterion is simple, while the proof is surprisingly difficult. A new technique had to be developed in order to prove it. 

Here and in the following we mean by a \emph{deterministic process} a stochastic process $(X,\bbp^x)_{x\in\bbr^d}=(X^x)_{x\in\bbr^d}$ which does not depend on $\omega$, i.e. there exists a function $f:\bbr^d \times [0,\infty[ \to \bbr^d$ such that 
\[
X_t^x(\omega)=f(x,t)
\]
for every $\omega\in\Omega$. Since it is always assumed that $\bbp^x(X_0=x)=1$ we write $X^x$ for the simple process $(X,\bbp^x)$ and call it a \emph{path} of the process. Since deterministic processes are adapted to every possible filtration, we do not mention it further but we assume that a fixed stochastic basis $(\Omega, \cf, \bbf=(\cf_t)_{t\geq 0}, \bbp)$ is always in the background. In constructing examples it is an advantage of deterministic processes that one does not have to care about the filtration. 

We are concerned with \emph{Markov processes} in the sense of Blumenthal and Getoor (cf. \cite{blumenthalget}) which are sometimes called universal Markov processes (cf. \cite{bauerwt}, \cite{niels3}) or Markov families. This is due to the fact that the examples we have in mind are related to concepts like the semigroup of the process, the generator or the martingale problem. With little effort most of the results can be transformed to the case of simple Markov processes (with only one starting point). The Markov processes $X=(X_t)_{t\geq 0}$ we are treating here are allowed to start in every point of the respective state space and furthermore, time homogeneity is present, i.e.,  
writing for $s,t\geq 0$, $x,y,z \in\bbr^d$ and a Borel set $B$ in $\bbr^d$ $P_{s,t}^x(z,B):=\bbp^x(X_{t}\in B | X_s=z)$ we have
\begin{align} \label{timehomone}
P_{t-s}(z,B):=P_{s,t}^x(z,B)=P_{s+h,t+h}^y(z,B), \hspace{10mm} h\geq 0.
\end{align}
For a deterministic Markov process, time homogeneity can be characterized as follows: if there exist $s,t\geq 0$ and $x,y\in\bbr^d$ such that $X_s^x=X_t^y$ we obtain
\begin{align} \label{timehom}
  X_{s+h}^x = X_{t+h}^y
\end{align}
for every $h\geq 0$. 
The class of deterministic Markov processes serves as a rich source of counterexamples (cf. \cite{detmp1} and the present paper). Furthermore they can be used as (parts of the) mean processes in stochastic volatility models or as interesting mathematical objects in their own right.

Since the definitions and notations for some of the classes of processes we are treating are not unified, let us first fix some terminology: a Markov process in the above sense, i.e. satisfying \eqref{timehomone}, is called \emph{Hunt process} if it is quasi-left continuous (cf. Definition I.2.25 of \cite{jacodshir}) with respect to every $\bbp^x$ $(x\in\bbr^d)$. For a deterministic process this is equivalent to ordinary left continuity of the paths. 

We say that a process $(X,\bbp^x)_{x\in\bbr^d}$ is a \emph{semimartingale}, if $X^x$ is one for every $x\in\bbr^d$.
A Markov semimartingale $X$ is called \emph{It\^o process} (cf. \cite{vierleute} and \cite{cinlarjacod81}) if it has characteristics $(B,C,\nu)$ of the form:
\begin{align*}
B_t^{(j)}(\omega)&=\int_0^t \ell^{(j)}(X_s(\omega)) \ ds & j=1,...,d\\
C_t^{jk} (\omega)&=\int_0^t Q^{jk} (X_s(\omega)) \ ds &j,k=1,...,d\\
\nu(\omega;ds,dy)&=N(X_s(\omega),dy)\ ds 
\end{align*}
where $\ell^{(j)},Q^{jk}:\bbr^d \pf \bbr$ are measurable functions, $Q(x)=(Q^{jk}(x))_{1\leq j,k \leq d}$ is a positive semidefinite matrix for every $x\in\bbr^d$, and $N(x,\cdot)$ is a Borel transition kernel on $\bbr^d \times \cb(\bbr^d \backslash \{0\})$.

The following diagram gives an overview on the interdependence of the classes of processes we are treating here:
\begin{align*}
\begin{array}{cccccccccc}
    \begin{array}{c}\text{rich} \\ \text{Feller} \end{array}& \subset& \text{It\^o}& \subset  & \begin{array}{c}\text{Hunt} \\ \text{ semimartingale} \end{array} &  & \subset & \begin{array}{c}\text{Markov} \\ \text{ semimartingale} \end{array}
         &\subset  & \text{semimartingale}     \\
                  \rule[5mm]{0mm}{0mm}    \cap & &  & & \cap & & & \cap \\
                  \text{Feller} & & \subset  &\rule[5mm]{0mm}{0mm} & \text{Hunt} & &\subset & \text{Markov}&
\end{array}
\end{align*}
In the present paper we are not too much concerned with Feller processes (only in Section 2, where we recall the definition). This is because we are interested in processes which exhibit jumps. A deterministic jump is a fixed time of discontinuity which Feller processes do not have.

We start with a simple construction principle which we proposed in \cite{detmp1} and which will be generalized in Section 3 below: 
\begin{example} \label{ex:startingpoint}
Let $R\subseteq \bbr$ and $\Phi:\bbr\to R$ be bijective and such that $\Phi(0)=0$. In this case a Markov process on $R$ is given by
\begin{align} \label{construction}
X_t^x(\omega):=\Phi(t+\Phi^{-1}(x)),\hspace{10mm} \text{ for every } \omega \in \Omega,
\end{align}
i.e. by shifting the function $\Phi$ to the left and to the right. By inverting the function $x\mapsto \Phi(t+\Phi^{-1}(x))$ we know where a path being at time $t$ in $z\in R$ has started at time zero: $x=\Phi(\Phi^{-1}(z)-t)$.

The function $\Phi$ will be called a generating path of the process $X$ since it contains all the information of the process on $R$ (cf. Definition \ref{def:genpath} below). Obviously the restriction $\Phi(0)=0$ is not needed and any shifted generating path $t\mapsto \Phi(t-s)$ ($s\in \bbr$) would have served as well. In general one needs infinitely many generating paths in order to describe a deterministic Markov process on $\bbr^d$.
\end{example}

Let us have a look at a first example which emphasizes how rude a process can be if we do not assume any regularity of the paths, though the $\Phi$ is even bijective on $\bbr$, i.e. the process is given by a single generating path.
\begin{example}
Let $\Phi(t):=(t+1) 1_{\bbr \backslash \bbq} (t) + t 1_\bbq (t)$ and use the construction principle described in Example \ref{ex:startingpoint}: we obtain the deterministic Markov process on $\bbr$
\[
X_t^x=\begin{cases} t+x+1 & \text{if } x\in \bbq \text{ and }    x+t\notin\bbq \\
                    t+x   & \text{if } x\in \bbq \text{ and }    x+t \in \bbq \\ 
                    t+x   & \text{if } x\notin \bbq \text{ and } x+t\notin\bbq \\
                    t+x-1 & \text{if } x\notin \bbq \text{ and } x+t \in \bbq.
      \end{cases}              
\]
For this process every path $t\mapsto X_t^x$ is discontinuous in \emph{every} point. The image of every path is dense in $[x-1,\infty[$.
\end{example}

In the deterministic world \emph{homogeneity in space} can be written as follows: 
\begin{align} \label{spacehom}
\text{for every } t\geq 0 \text{ and } x,z\in\bbr^d \text{ we have } X_t^{x+z}=X_t^x+z
\end{align}
The process above is not homogeneous in space since
\[
X_{\sqrt{2}}^{-\sqrt{2}} +\sqrt{2}=\sqrt{2} - 1 \neq \sqrt{2} + 1 =X_{\sqrt{2}}^0 
\]
but it is a natural question whether there are Markov processes which are homogeneous in space and time, but which are not L\'evy processes. It turns out that the axiom of choice is needed in order to construct such processes. We have not found a proof for the following result in the literature, but somehow it belongs to the folklore of our subject. The proof it not too difficult and hence left to the reader (compare in this context \cite{hamel}). 

\begin{proposition}
Let $X$ be a deterministic Markov process. $X$ is homogeneous in time \emph{and} space, iff $t\mapsto X_t^0$ is a $\bbq$-homomorphism on the $\bbq$-vector space $\bbr$ and the other paths are given by $X_t^x=X_t^0+x$. 
\end{proposition}

Deterministic L\'evy processes form a one-dimensional subspace of the infinite-dimensional $\bbq$-vector space of deterministic Markov processes which are homogeneous in space and time.

\begin{example}
In order to get a concrete example of a deterministic Markov process which is homogeneous in space and time but not a L\'evy process, let $B= \{1\}\cup \{b_i:i\in I \} $ be (an uncountable) algebraic $\bbq$-basis of $\bbr$ and set $\Phi(1)=1$, $\Phi(b_i)=2b_i$ for every $i\in I$. The $\Phi$ defined as such is bijective since for every $y\in\bbr$ written as 
\[
q_0 1 + q_1 b_1 + ...+ q_n b_n
\]
we obtain 
\[
q_0 1 + \frac{q_1}{2}b_1 + ... + \frac{q_n}{2} b_n
\]
as its pre-image under $\Phi$.
\end{example}

In order to avoid pathological examples as above we will assume in the following that the processes we encounter are \emph{c\`adl\`ag}.

Let $(X^x)_{x\in\bbr^d}$ be a deterministic Markov process and let $x\in\bbr^d$. Immediately by \eqref{timehom} we obtain: if there exists $t_0<t_1$ such that $X_{t_0}^x=X_{t_1}^x$ then $X_{t_0+h}^x=X_{t_1+h}^x$ for every $h\geq 0$, i.e. if a path returns to a point, it has visited before, it becomes periodic. In order to describe $X^x$ and every path starting in the range $R$ of $X^x$ one could use $\Phi(t):=X_t^x$. This shows that we can be a bit more general by allowing generating paths which do become periodic in a certain sense. This is analyzed in detail in Section 3 where we solve our initial question III). While in the Hunt case a countable number of generating paths was sufficient, we show that in general an uncountable number of generating paths is needed in the presence of jumps. The case of `well behaved' paths leads to a different perspective on the classification theorem for Hunt processes (cf. \cite{detmp1} Theorem 2.11).

Let us give a brief outline on how the paper is organized: in the subsequent section we treat the question of expandability, i.e. given a path, does there exist an element of a certain class of processes containing this path. Section 3 deals with the general structure of such a process while in Section 4 we prove a criterion when a deterministic Markov process is a semimartingale. Section 5 is devoted to the more general case of time inhomogeneous processes. Some examples complementing those given in the text are collected in Section 6. The so called space-time Markov process is included here. Our main results are Theorem \ref{thm:markovexp} and the closely related Theorems \ref{thm:semimg} and \ref{thm:analysis}.

Most of the notation we are using is more or less standard. Vectors are column vectors and $'$ denotes a transposed vector or matrix. The $j$-th entry of the vector $v$ is $v^{(j)}$. For a set $A\subseteq\bbr$ and $n\in\bbn$ we write $A+n:=\{r\in\bbr:r=a+n \text{ for an } a\in A\}$. Note that we prefer to write $]s,t[$ for an open interval rather than $(s,t)$ and use the same convention for semi-open intervals. In the context of semimartingales we follow mainly \cite{jacodshir}, in the context of Feller processes \cite{revuzyor} and \cite{niels3}.

\section{Expandability}

In the construction of counterexamples one is often concerned with only one path rather than with a universal process, i.e. one considers a single starting point $x$ and constructs a path which meets certain requirements. However, one is still interested if this path can really be a part of a process in the desired class. It is important for us to consider universal processes, since the definition of e.g. a Feller process does not make sense otherwise. \\
In this section we answer the question whether a given right-continuous function $f:[0,\infty[\to \bbr^d$ is a path of a deterministic Markov process or of one of its subclasses. By time homogeneity one automatically knows how the process has to behave starting from any point within the range of $f$. Outside of that range one can usually set $X_t^x=x$ for every $t\geq 0$. The only exception to this rule are Feller processes (see below).

\begin{definition}
A right-continuous function $f:[0,\infty[\to \bbr^d$ is \emph{expandable} to an element of a certain class of processes $\cc$ if there exists a deterministic process $X\in\cc$ such that $X^{f(0)}_t=f(t)$ for every $t\geq 0$. We write $\cc$-expandable for short.
\end{definition}

By the same reasoning as in the proof of Theorem 2.7 of \cite{detmp1} we obtain that a path of a \emph{one-dimensional} deterministic Markov process can not move continuously to a point it has visited before, i.e. either the function becomes constant at a time $t_1$ or it jumps to a point it has visited before. The two cases are illustrated below.
\begin{center}
\includegraphics[width=5cm, angle=270]{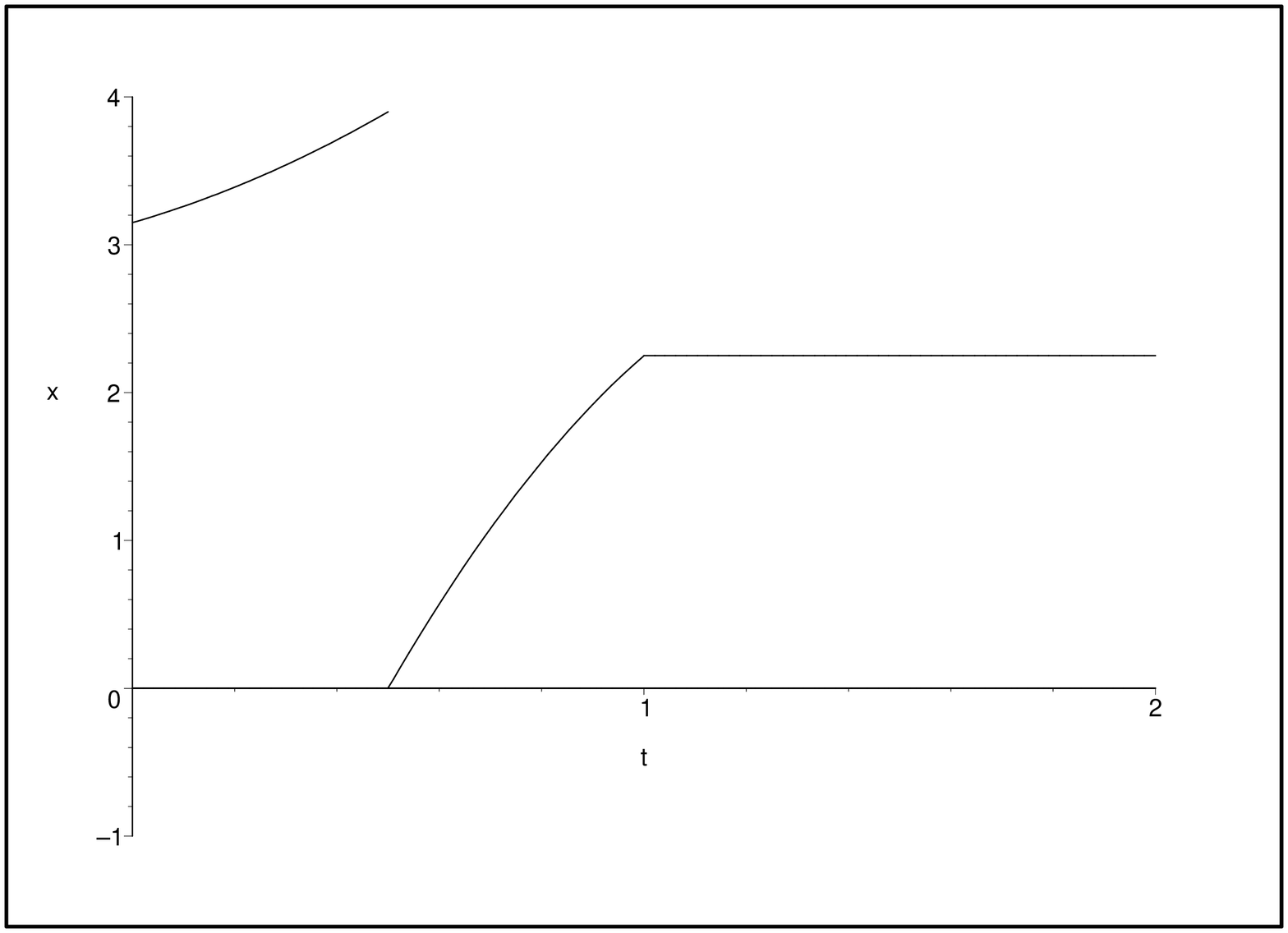}
\includegraphics[width=5cm, angle=270]{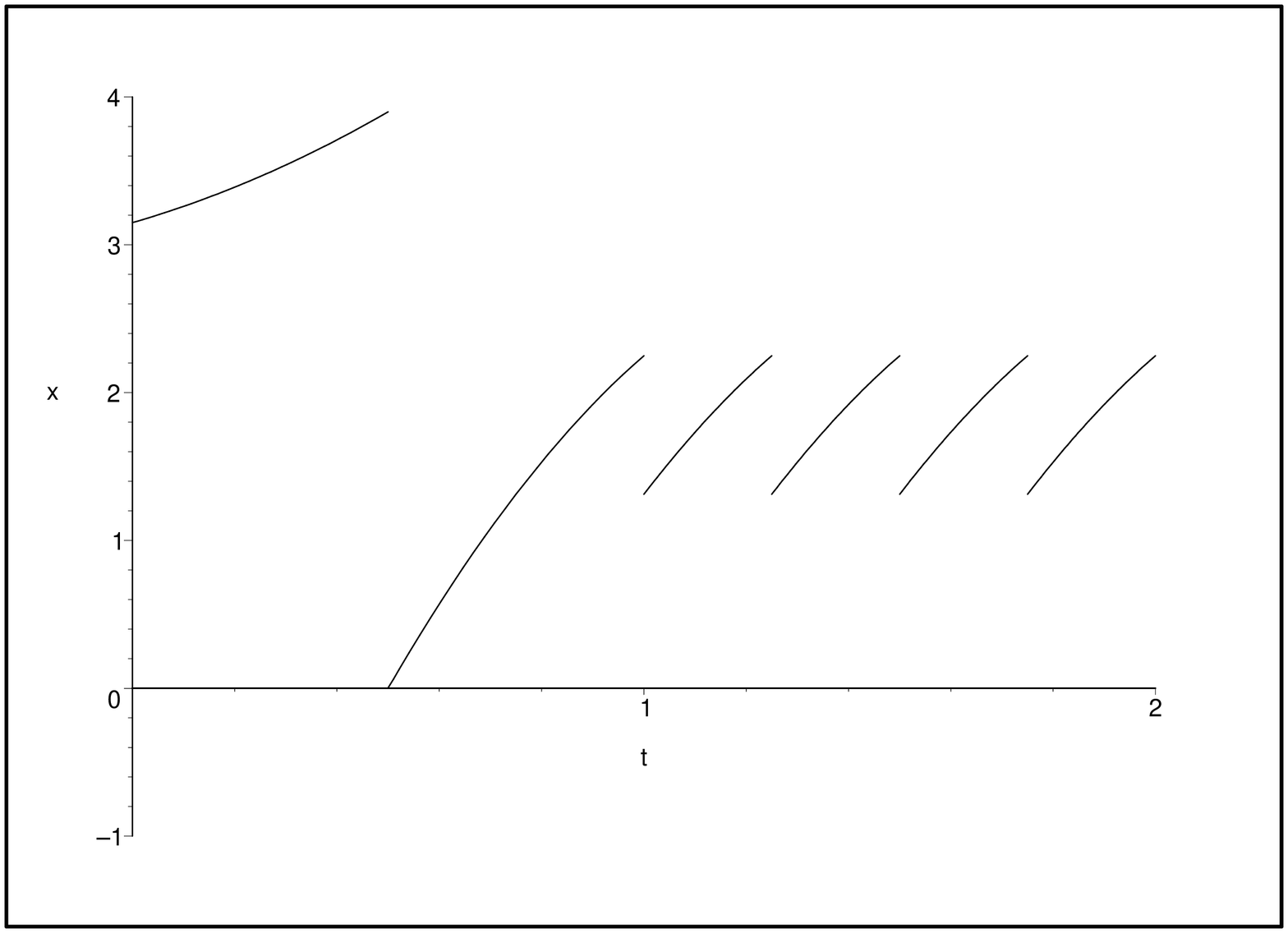}
\end{center}
Therefore, the key is the following kind of periodicity.

\begin{definition} \label{def:tzero}
Let $f:[0,\infty[\to\bbr^d$ be right-continuous. We define
\[
t_1:=\inf\{t\geq 0 : \text{ there exists } s<t \text{ such that } f(s)=f(t)\}
\] 
and
\[
t_0:=\inf\{s\geq 0 : \text{ there exists } t>s \text{ such that } f(s)=f(t)\}
\]
(a) $f$ is called \emph{jump-periodic} if the following holds: $t_0<t_1$ and for every $u\geq t_1$ we have $f(u)=f(u-(t_1-t_0))$. \\
(b) $f$ is called \emph{finally constant}, if there exists a $t_0\geq 0$ such that $f$ is injective on $[0,t_0[$ and constant on $[t_0,\infty[$. (Hence $t_0=t_1$)\\
(c) Let $f$ be jump-periodic. The function $f^{\leftarrow}: f([0,\infty[) \to [0,t_1[$ given by $f^{\leftarrow}(y):=\inf\{t\geq 0 : f(t)=y\}$ is called \emph{generalized inverse} of $f$.
\end{definition} 

\begin{remark}
(a) In the finally constant case one could speak of a period of length zero, in the jump-periodic case of length $t_1-t_0$. If a path never visits a point it has visited before we obtain a period of length infinity. \\
(b) Not every periodic function is Markov-expandable (e.g. $t\mapsto \sin(t)$). However, for a jump-periodic function $f$ the shifted function $t\mapsto f(t-t_0)$ is periodic. \\
(c) By the right-continuity of the function all infima of the above definition are minima, except for $\inf \emptyset = \infty$. \\ 
(d) The notions finally constant and jump-periodic are transfered easily to generating paths, i.e. functions on intervals $I\subseteq \bbr$ (see Definition \ref{def:genpath} below). Obviously this makes only sense if the right endpoint of $I$ is $\infty$.\\
(e) The notion jump-periodic might be misleading in two or more dimensions, because there does not have to be a jump in order to reach a point the process has visited before. An example of this kind is given by $f(t):=(\sin(t),\cos(t))'$. 
\end{remark}

\begin{theorem} \label{thm:markovexp}
A function $f:[0,\infty[\to \bbr^d$ is Markov-expandable iff it is injective or finally constant or jump-periodic. 
\end{theorem}

\begin{proof}
Let $f$ be Markov-expandable, i.e. there exists a deterministic Markov process such that $X_t^{f(0)}=f(t)$. If this path is not injective, there exists points $s<t$ such that $f(s)=f(t)$. Now let $t_0\leq t_1<\infty$ be defined as above. If $t_0=t_1$ the path is locally constant on an interval $[t_0,t_0+\varepsilon[$ and hence constant on $[t_0,\infty[$. Otherwise we obtain $f(t_0+h)=f(t_1+h)$ for every $h\geq 0$ by time homogeneity. Since every $u\geq t_1$ can be written as $u=t_1+h$ with $h\geq 0$ and since $u-(t_1-t_0)=t_0+h$ we obtain $f(u)=f(u-(t_1-t_0))$. \\
Now let one of the three properties be fulfilled.
\emph{Case 1:} $t_0=t_1=\infty$. In this case the construction described in \ref{ex:startingpoint} works, i.e. take $f$ as generating path on $f([0,\infty[)$. Outside of the range of $f$ we set $X_t^x=x$ for every $t\geq 0$.\\
\emph{Case 2:} $t_0=t_1<\infty$. The function becomes constant at $t_0$, i.e. $f(t)=c$ for $t\geq t_0$. In this case there is only one point which is visited more than once. The construction works as in Case 1, the only difference being that $f$ is inverted only on $[0,t_0[$. Paths hitting $c$ become constant. In particular we have $X_t^c=c$ for every $t\geq 0$. Time homogeneity is then clear and the Markov property is trivially fulfilled since $X_t^x=X_t^y$ for $x\neq y$ happens only in the case $X_t^x=c=X_t^y$ and this case we have $X_{t+h}^x=c=X_{t+h}^y$ for every $h\geq 0$.\\ 
\emph{Case 3:} $t_0<t_1<\infty$. Set 
\[
X_t^x:=\begin{cases}f(t+f^{\leftarrow}(x)) & \text{if } x\in f([0,\infty[) \\
                    x                     & \text{else.}
                \end{cases}
\]
Now let $z,w\in\bbr^d$, $t,h\geq 0$ and $x\in f([0,\infty[)$ such that $X_t^x=z$. For $t<t_1$ we know where the path being at time $t$ in $z\in f([0,\infty[)$ has started at time zero:
\[
x=f(f^\leftarrow(z)-t)
\]
if $t\leq f^\leftarrow(z)$. Otherwise (i.e. $f^\leftarrow(z)<t<t_1$) either $f^\leftarrow(z) \leq t_0$ and there is no $x$ such that $X_t^x=z$ or $t_0< f^\leftarrow(z) <t<t_1$ and we have $x=f(f^\leftarrow(z)-t+t_1-t_0)$. \\
The Markov property is clear, since by definition we have for every $t\geq 0$ that $x\neq y$ implies $X^x_t\neq X^y_t$. 
By jump-periodicity we can subtract $n\cdot (t_1-t_0)$ for a suitable $n\in\bbn$ `inside' of $f$ and can therefore assume w.l.o.g. $t<t_1$. We obtain:
\begin{align*}
P_{t,t+h}^x(z,\{w\}) = 1 &\gdw f\Big((t+h)+f^\leftarrow(x)\Big) = w \\
&\gdw f\Big((t+h)+f^\leftarrow(f(f^\leftarrow(z)-t))\Big) = w \\
&\gdw f\Big(h+f^\leftarrow(z) \Big)= w \\
&\gdw P_{0,h}^0(z,\{w\}) = 1.
\end{align*}
This yields the homogeneity in time. Finally we have $X_t^{f(0)}=f(t)$ since $f^{\leftarrow}(f(0))=0$.
\end{proof}

Virtually all examples of homogeneous diffusions (with jumps) in the sense of \cite{jacodshir} in the literature are Markov processes. As a first application of the theorem we prove the existence of a homogeneous diffusion which is not Markovian.

\begin{example} \label{ex:hdwjunif} We define the following deterministic process:
denote for $t\in\bbr$ the sabertooth function $\{t\}:=t-\left\lfloor t \right\rfloor$ and define
\begin{align*}
f(t)&:=\binom{\{t\}}{0} \cdot 1_{[0,1/6[} \left( \left\{ \frac{t}{6} \right\} \right)
     +\binom{1-\{t\}}{\{t\}} \cdot 1_{[1/6,2/6[} \left( \left\{ \frac{t}{6} \right\} \right)
     +\binom{0}{1-\{t\}} \cdot 1_{[2/6,3/6[} \left( \left\{ \frac{t}{6} \right\} \right) \\
     &+\binom{-\{t\}}{0} \cdot 1_{[3/6,4/6[} \left( \left\{ \frac{t}{6} \right\} \right)
     +\binom{\{t\}-1}{-\{t\}} \cdot 1_{[4/6,5/6[} \left( \left\{ \frac{t}{6} \right\} \right)
     +\binom{0}{\{t\}-1} \cdot 1_{[5/6,1[} \left( \left\{ \frac{t}{6} \right\} \right)
\end{align*}
We define the process $X$ by
\[
X_t^x:=\begin{cases}f(t+f^{\leftarrow}(x)) & \text{if } x\in f([0,\infty[) \\
                    x                     & \text{else.}
                \end{cases}
\]
For the readers convenience we include the following graphic which shows what the trajectory of the process looks like:

\centerline{
\setlength{\unitlength}{1cm}
\begin{picture}(1,3)
\put(0,0){\line(0,1){3}}
\put(-2,1){\line(1,0){4.1}}
\linethickness{0.5mm}
\put(0,1){\vector(0,1){1}}
\put(0,2){\vector(-1,-1){1}}
\put(-1,1){\vector(1,0){1}}
\put(0,1){\vector(0,-1){1}}
\put(0,0){\vector(1,1){1}}
\put(1,1){\vector(-1,0){1}}
\put(1,0.8){1}
\put(2,0.8){2}
\put(-0.2,1.8){1}
\end{picture}
}
This process can be written as
\begin{align} \label{inteq}
X_t=x+\int_0^t \ell(X_s) \, ds 
\end{align}
for every starting point $x\in\bbr^2$ where $\ell:\bbr^2\to\bbr^2$ is given by
\[
\ell(y)=\begin{cases} ( \sign(y^{(1)}), 0)' & \text{ if } y^{(1)}=0 \text{ and } \abs{y^{(2)}}\leq 1\\
                      ( 0, -\sign(y^{(2)}))' & \text{ if } y^{(2)}=0 \text{ and } \abs{y^{(1)}} \leq 1 \\
                      ( -\sign(y^{(1)}), \sign(y^{(2)}))' & \text{ if } \abs{y^{(1)}-y^{(2)}}=1 \text{ and } \abs{y^{(1)}} \leq 1\\
                      (0,0)'                & \text{ else.}
\end{cases}
\]
Therefore the process is a homogeneous diffusion, for which $\ell$ does not depend on the starting point, but by Theorem \ref{thm:markovexp} it is not a Markov process.

\end{example}

\begin{remark}
The integral equation \eqref{inteq} has infinitely many solutions. Out of these only 2 solutions are Markovian. Namely the one where the process reaching $(0,0)'$ always goes `up' and the one which always goes `down' when reaching this point.
\end{remark}

\begin{corollary} 
A function $f:[0,\infty[\to \bbr^d$ is Hunt-expandable iff it is Markov expandable and continuous. 
\end{corollary}

The order structure of $\bbr$ gives rise to the following characterization in the one-dimensional case. Compare in this context \cite{detmp1} Theorem 2.7.

\begin{corollary} \label{cor:huntexpone}
A function $f:[0,\infty[\to \bbr$ is Hunt-expandable iff it is continuous and there exists a $t_0\in [0,\infty]$ such that $t\mapsto f(t)$ is strictly monotone (increasing or decreasing) on $[0,t_0[$ and constant on $[t_0,\infty[$.
\end{corollary}

Since $t_0\in [0,\infty]$ the `pure types' of paths which are only constant \emph{or} strictly monotone are included in this corollary.

\begin{proposition}
A function $f:[0,\infty[\to \bbr^d$ is It\^o-expandable iff it is Hunt-expandable and in addition each component of the vector $f-f(0)$ is absolutely continuous w.r.t. Lebesgue measure. 
\end{proposition}

\begin{proof}
For a deterministic Hunt semimartingale the second and third characteristic vanish. Therefore the first characteristic can be written as $B^x_t=X^x_t-X^x_0$ on the image of $f$ where $X^x_t$ is just $f$ shifted to the left. The absolute continuity is inherited from $f$. Anywhere else we set $X_t^x=x$ for every $t\geq 0$, hence $B^x_t=0$ which is as well absolutely continuous. The density $\tilde{b}$ we obtain is deterministic and hence optional w.r.t. any filtration (Compare in this context \cite{vierleute} Theorem 6.25 and the introduction to Section 7 of that article).
\end{proof}

An important subclass of Hunt processes are Feller processes. If the semigroup $(T_t)_{t\geq 0}$ of operators which is defined on the bounded Borel measurable functions by
\[
T_t u(x) = \bbe^x u(X_t) = \int_\Omega u(X_t(\omega)) \, \bbp^x(d\omega)=\int_{\bbr^d} u(y) \, P_t(x,dy).
\]
satisfies the following conditions\\
\hspace*{5mm}$(F1)$ $T_t:C_\infty(\bbr^d) \to C_\infty(\bbr^d)$ for every $t\geq 0$, \\
\hspace*{5mm}$(F2)$ $\lim_{t\downarrow 0} \norm{T_tu-u}_\infty =0$ for every $u\in C_\infty(\bbr^d)$.\\
we call the process (and the semigroup) \emph{Feller}.

The following proposition only holds in the one-dimensional case.

\begin{proposition} \label{prop:Fellerexp}
A function $f:[0,\infty[\to \bbr$ is Feller-expandable iff it is Hunt-expandable.
\end{proposition}

\begin{proof}
The `only if'-part is clear. Now let $f$ be Hunt-expandable. W.l.o.g. let $f$ be increasing up to time $t_0$ (cf. Corollary \ref{cor:huntexpone}). Set 
\[
\Phi(t):=\begin{cases} f(t) & \text{if } t\geq 0 \\
                      f(0)+t & \text{if } t <0
        \end{cases}
\]
and
\[
X_t^x:=\begin{cases}\Phi(t+\Phi^{\leftarrow}(x)) & \text{if } x\in \Phi(]-\infty,\infty[) \\
                    x                     & \text{else.}
                \end{cases}
\]
The image of $\Phi$ is either $]-\infty,\infty[$ or there exists an $a\in\bbr$ such that the image is $]-\infty,a[$ or $]-\infty, a]$ in the latter cases we set $X_t^x:=x$ for every $t\geq 0$. It is simple to show that (F1) and (F2) are fulfilled by this process. Compare in this context Theorem 3.5 of \cite{detmp1}.
\end{proof}

In the general case we have the following:

\begin{proposition} \label{prop:Fellerexpddim}
If a function $f:[0,\infty[\to \bbr^d$ is continuous and finally constant or jump-periodic it is Feller-expandable.
\end{proposition}

\begin{proof}
The property (F2) is trivially fulfilled. The whole information is contained in the restriction $f|[0,t_1]$. Since the image $f([0,t_1])$ is compact, vanishing at infinity is not a problem. In fact the only problems which could occur are by destroying continuity. However, since $f([0,t_1])$ is closed, those problems could only show up `within' the curve. Now assume that there was a sequence $(x_n)_{n\in\bbn}\subseteq f([0,t_1])$ and an $h>0$ such that
\[
X_{t_n}=x_n\xrightarrow[n\to\infty]{} x \text{ and } (X_{t_n+h})_{n\in\bbn} \text{ is divergent.}
\]
We have $x\in f([0,t_1])$. By compactness we have that there are two points $y\neq z$ and subsequences $(n(k))_{k\in\bbn}$ and $(n(j))_{j\in\bbn}$ such that
\[
X_{t_{n(k)}+h} \rightarrow y \neq z \leftarrow X_{t_{n(j)}+h}
\]
Since $t_{n(k)}$ and $t_{n(j)}$ are both restricted to the compact interval $[0,t_1]$ there exists convergent sub-subsequences such that
\begin{align*}
(t_{n(k(l))})\to t \text{ and } X_{t_{n(k(l))}+h} \to y \text{ and } X_{t_{n(k(l))}} \to x \\
(t_{n(j(i))})\to t \text{ and } X_{t_{n(j(i))}+h} \to y \text{ and } X_{t_{n(j(i))}} \to x
\end{align*} 
By continuity of $f$ we obtain $X_t=x$ and $X_{t+h}=y$ while $X_s=x$ and $X_{s+h}=z$. This contradicts time homogeneity.
\end{proof}

We complement the previous two propositions by an example of a two-dimensional Hunt-expandable functions which is not Feller-expandable.

\begin{example}
The following function $f:[0,\infty[\to\bbr^2$ is Hunt-expandable but not Feller-expandable:
\begin{align*}
f(t):=\sum_{n\in\bbn_0} & \Bigg(\left( \binom{1-\frac{1}{2^n}}{0} +\binom{0}{(-1)^n} (t-2n) \right) 1_{[2n,2n+1[}(t)  \\ 
&+\left( \binom{1-\frac{1}{2^n}}{(-1)^n} +\binom{\frac{1}{2^{n+1}}}{(-1)^{n+1}} (t-(2n+1)) \right) 1_{[2n+1,2n+2	[}(t) \Bigg)
\end{align*}
For every $u\in C_\infty(\bbr^2)$ which coincides with the projection on the second component on $[0,1]\times [-1,1]$ we obtain
\[
x_n:=X_{2n}^0=\binom{1-\left(\frac{1}{2} \right)^n }{0} \xrightarrow[n\to\infty]{}\binom{1}{0}
\]
but
\[
u(X_{2n+1}^0)=\begin{cases} 1 & \text{if } n \text{ is even,} \\
                            -1& \text{else.}
              \end{cases}
\]
Therefore $T_1(u(x_n))=u(X_1^{x_n})=u(X_{2n+1}^0)$ does not converge.
\end{example}

The generator $(A,D(A))$ is the strong right-hand side derivative of the semigroup $(T_t)_{t\geq 0}$ at zero.
A Feller process is called \emph{rich} if the domain $D(A)$ of this closed operator contains the test functions $C_c^\infty(\bbr^d)$. In \cite{mydiss} it was shown that every rich Feller process is an It\^o process. By Proposition 3.8 of \cite{detmp1} and the above reasoning we obtain the following:

\begin{proposition} A function $f:[0,\infty[\to \bbr$ is rich-Feller-expandable iff it is continuously differentiable on $]0,\infty[$, right-differentiable at zero and if there exists a $t_0\in [0,\infty]$ such that $\partial_t f(t) \neq 0$ on $[0,t_0[$ and $\partial_t f(t) = 0$ on $]t_0,\infty[$.
\end{proposition}

Rich Feller processes are interesting, because by a classical result which is due to Courr\`ege \cite{courrege} one knows that their generator $A$ is a pseudo-differential operator. The symbol of this operator contains a lot of information about the corresponding process (cf. \cite{schilling98}, \cite{schillingschnurr}).

\begin{remark} For a function $f:[0,\infty[\to \bbr^d$ to be L\'evy-expandable it is necessary and sufficient to be of the form $t\mapsto f(0)+c\cdot t$ with $c\in\bbr^d$. This follows directly from the L\'evy-It\^o decomposition (cf. e.g. \cite{protter} Theorem I.42).
\end{remark}

\section{The General Structure of a Deterministic Markov Process}

Our starting point is the following result for Hunt processes (see \cite{detmp1} Section 2). We will see in the following that the generalization for deterministic processes with jumps is not straight forward and a new approach is needed. We will work one-dimensional since even in the Hunt case there is no possibility to generalize the result to higher dimensions. 

\begin{theorem} \label{thm:huntstructure}
A family of functions $t\mapsto X^x_t$, each mapping $[0,\infty[$ into $\bbr$, is a deterministic Hunt process if and only if there exists a decomposition of $\bbr$ into disjoint ordered intervals $(J_j)_{j\in Z}$ where $Z\subset \{-n,...,0,...,m\}$ with $n,m\in\bbn\cup \{\infty \}$ such that
on every Interval $J_j$ the functions $t\mapsto X^x_t$ (for $x\in J_j$) are either all constant or there exists a continuous generating path $\Phi_j:I_j\to J_j$ for $J_j$ which is surjective and either strictly monotonically increasing or strictly monotonically decreasing and such that
  \[
    X_t^x=\Phi_j(t+\Phi_j^{-1}(x)) \text{ for } x\in J_j \text{ and } t\in [0,\infty[ \, \cap \, (I_j-\Phi_j^{-1}(x))
  \]
and the $I_j$ are intervals containing zero.
\end{theorem}

With every interval $J_j$ we associate the \emph{type} $\oplus$, $\ominus$ resp. $\odot$, if $\Phi$ is increasing, decreasing resp. the process is constant on the interval. This allows to describe the process from an abstract point-of-view as a sequence like $...|\odot|\oplus|\odot|\ominus|...$. This sequence is called the \emph{structure} of the process.
To emphasize that the upper endpoint of the lower interval belongs to the lower (resp. higher) interval we write $\odot]\oplus$ (resp. $\odot[\oplus$). The behavior of the paths of the deterministic Hunt process is totally described by the decomposition $(J_j)_{j\in Z}$ and the sequence of generating paths. Compare in this context Remark 2.13 of \cite{detmp1}. In particular we described there which structures are forbidden and by which conventions one gets a unique representation.

The theorem above can be used to construct deterministic Markov processes with jumps. A first idea is to just add a `jump structure' to a given Hunt process, i.e. a sequence $(a_n, b_n)$ with $a_n\in\bbr, b_n\in\bbr$ meaning that if 
\[
\lim_{s\uparrow t} X_s= a_n \text{ then } X_t = b_n  
\]
from the point $b_n$ the process moves on as the underlying Hunt process. In addition we could allow the process to jump from $a_n$ to $b_n^+$ resp. $b_n^-$ depending on whether the path was increasing before reaching $a_n$ or decreasing. Let us remark that this does not contradict the Markov property, since the process does not actually reach $a_n$. For a (one-dimensional) Hunt process there exist only the possibility to reach a point from above \emph{or} from below. Using such a jump structure one can construct several interesting examples of Markov processes. 

It would be nice if every deterministic Markov process could be described by the above construction. However, this is not the case. The following example shows that we do not even necessarily have an interval of monotonicity before a jump:
\begin{example} \label{ex:AtoD}
Let
\[
X_t^1:=(1-t)1_A(t)+(t-1) 1_B(t)+2\cdot 1_{\{1\}}(t) + (1+t) 1_C(t) + (3-t)1_D(t) + 3\cdot 1_{[2,\infty[}(t)
\]
where
\begin{align*} 
A:=& \bigcup_{n\in (2\bbn_0)} \left[ \frac{2^n-1}{2^n}, \frac{2^{n+1}-1}{2^{n+1}} \right[ & 
B:=&\bigcup_{n\in (2\bbn_0+1)} \left[ \frac{2^n-1}{2^n}, \frac{2^{n+1}-1}{2^{n+1}} \right[ \\
C:=& \bigcup_{n\in (2\bbn_0+1)} \left[ 1+\frac{1}{2^{n+1}}, 1+\frac{1}{2^{n}} \right[ &
D:=&\bigcup_{n\in (2\bbn_0)} \left[ 1+\frac{1}{2^{n+1}}, 1+\frac{1}{2^{n}} \right[
 \end{align*}
This path looks as follows:

\centerline{
\setlength{\unitlength}{1cm}
\begin{picture}(1,4)
\put(0,0){\vector(0,1){4}}
\put(0,1){\vector(1,0){2.1}}
\linethickness{0.5mm}
\put(0,2){\line(1,-1){0.5}}
\put(0.5,0.5){\line(1,1){0.25}}
\put(0.75,1.25){\line(1,-1){0.125}}
\put(0.875,0.875){\line(1,1){0.0625}}
\put(0.9375,1.0625){\line(1,-1){0.03125}}
\put(0.96875,0.96875){\line(1,1){0.015625}}
\put(1,0.8){1}
\put(2,0.8){2}
\put(-0.2,1.8){1}
\put(2,2){\line(-1,1){0.5}}
\put(1.5,3.5){\line(-1,-1){0.25}}
\put(1.25,2.75){\line(-1,1){0.125}}
\put(1.125,3.125){\line(-1,-1){0.0625}}
\put(1.0625,2.9375){\line(-1,1){0.03125}}
\put(1.03125,3.03125){\line(-1,-1){0.015625}}
\end{picture}
}
\end{example}
While the jumps before time 1 can be described by e.g. `reaching 1/2 from above, jump to -1/2' a description of this kind can not be given for the jump at time 1. 
This example shows that a deterministic Markov process is more than a deterministic Hunt process plus `jump structure'. The only way to describe this class of processes seems to be directly via the generating paths which do admit jumps and might look quite ugly as the example above shows. In Section 6 we include two more nasty examples. Example \ref{ex:loopjumps} shows that it is possible that from one starting point one might reach the same point $a_n$ in different ways (which are difficult to describe since there is no interval of monotonicity) resulting in different `landing points' $b_n^k$ $(k=1,2)$. This shows that even a start-point-dependent jump structure is not sufficient. In the case of Example \ref{ex:infinitelandingpoints} we have even infinitely many `landing points' $b_n^k$ $(k\in\bbn)$ for one $a_n$.  

Since it is not enough to use the generating paths of a Hunt process plus a jump structure we have to generalize the notion of generating paths itself, in order to capture the behavior of non-Hunt paths:

\begin{definition} \label{def:genpath}
Let $X$ be a deterministic Markov process and $I$ be an interval of the form $]a,\infty[$ respectively $[a,\infty[$ (with $a\in [-\infty,\infty[$ respectively $a\in ]-\infty,\infty[$). A surjective function $\Phi:I\to J$ is called a \emph{generating path} of $X$ on $J\subseteq \bbr^d$, if it is injective or jump-periodic or finally constant and \eqref{construction} holds for every $x\in J$ and $t\in [0,\infty[ \, \cap \, (I-\Phi^{-1}(x))$.
\end{definition}

\begin{remark}
Let us emphasize why generating paths are useful: think of a one-dimensional deterministic L\'evy process which is increasing. Using a function $f(y)=cy$, defined on $[0,\infty[$ as in the previous section, we can only describe the process for $x\in[0,\infty[$ by shifting this function to the left. In order to describe the whole process, we would need a sequence of such functions like: $f_n(y)=-n+cy$ (describing the process on $[-n,\infty[$). In contrast to this, \emph{one} generating path $\Phi(y):= cy$ defined on $\Phi:]-\infty,\infty[$ describes the whole process by shifting it to the left \emph{and} right.
\end{remark}

Even with the examples \ref{ex:AtoD}, \ref{ex:loopjumps} and \ref{ex:infinitelandingpoints} in mind one could still hope that a \emph{countable} number of generating paths is sufficient in order to describe the whole process. The following example shows that this is not the case, even if the process is a semimartingale with bounded increasing paths.

\begin{example} \label{ex:cantor}
Let $C$ be the Cantor set. It is well known that
\[\tag{$\star$}
C=\left\{ x\in [0,1]: \text{ there exists a representation } x=\sum_{j\in\bbn} a_j \left(\frac{1}{3}\right)^j \text{ with } a_j\in\{0,2\} \right\}
\]
and that every $y\in [0,1]$ can be written as $y=\sum_{j\in\bbn} b_j (1/3)^j$ with $b_j\in\{ 0,1,2\}$.
These representations are unique up to 0- respective 2-periods:
\[
\frac{1}{3}^j=2 \left(\frac{1}{3} \right)^{j+1} + 2 \left(\frac{1}{3} \right)^{j+2}+...
\]
Let $\widetilde{C}\subseteq C$ be all points in the Cantor set with a unique representation. Obviously $\widetilde{C}$ is uncountable since $C\backslash \widetilde{C}$ is countable. For every $x\in\widetilde{C}$ infinitely many $a_j$s are 0 and infinitely many $a_j$s are 2. $\widetilde{C}$ will be the set of our starting points. Now we define for every $x\in\widetilde{C}$ a path $X^x$ such that the paths have no common points: let $x\in \widetilde{C}$, $(q_n)_{n\in\bbn}\subseteq \bbq_+^*$ be a denumeration of $\bbq_+^*=\bbq\cap ]0,\infty[$ and $(h^x_n)_{n\in\bbn}\subseteq \bbn$ be a denumeration of the indices $j$ in the representation $(\star)$ such that $a_j=0$. Set
\[
X_t^x:=x+\sum_{n\in\bbn} \left(\frac{1}{3} \right)^{h^x_n} 1_{[q_n,\infty[}(t).
\]
Since the convergence of the series is is uniform, the function $t\mapsto X_t^x$ is c\`adl\`ag and by definition the function is strictly increasing (the function jumps upwards in every $q\in\bbq_+^*$). In particular the process is a semimartingale which is bounded since $X_t^x$ reaches only values in $[0,1]$. To be precise, $t\mapsto X_t^x$ takes only values $y$ which can be represented in the following way:
\[
y=\sum_{j\in\bbn} c_j \left( \frac{1}{3} \right) ^j \text{ such that } \begin{cases} c_j=2 & \text{if } a_j=2 \\ c_j\in \{0,1\} & \text{if } a_j=0. \end{cases}
\]
Since the representation $(\star)$ is unique (up to the excluded case of periods), we obtain that the paths are disjoint. In every point $z$ which is not reached by one of the paths we set $X_t^z=z$ for every $t\geq 0$. 
\end{example}

The example shows that the structure of a deterministic Markov process might be quite complicated. If we want to describe this structure we can not restrict ourselves to a countable number of generating paths (which is sufficient in the Hunt case). We can reduce ourselves to a subset of the set of all paths such that every point is contained in one of them, since time homogeneity is in place. Such a representation is by no means unique.   

\begin{remark}
A countable number of generating paths is sufficient, if the process is well behaved in the following sense: the image $f^x([0,\infty[)$ of every path can be written as the union of intervals $I^x_i$ such that there exist $t_1<t_2$ with
\[
f^x([t_1,t_2[) = I^x_i
\]
where $t_1$ is a jump time and $t_2$ is a jump time or the time where the path becomes constant and there is no jump time between $t_1$ and $t_2$. Though this definition of well behaved seems to be rather restrictive, it is met by every example in this paper except of \ref{ex:cantor} (and those of Section 1 which are not c\`adl\`ag). Since the paths have to be continuous between the times $t_1$ and $t_2$ defined as above we  obtain by time homogeneity that we have for $x\neq y$ and indices $i,j$
\[
I^x_i \cap I^y_j= \emptyset \text{ or } I^x_i \subseteq I^y_j \text{ or } I^y_j \subseteq I^x_i.
\]
This allows us to take a system $S$ of disjoint intervals such that every $I^x_i$ is either contained in $S$ or it is a subset of an element of $S$. This system is countable since every union of disjoint intervals of $\bbr$ is. On every such interval we get a generating path as in the proof of Theorem 2.11 of \cite{detmp1}. 
\end{remark}

\section{The Semimartingale Property}

In the present section we characterize the semimartingale property of a one-dimensional deterministic Markov process and show that such a result can not hold in dimensions $n\geq 2$. The following classical result is adapted from \cite{jacodshir} Proposition I.4.28:

\begin{proposition} \label{prop:detsemimg}
Let $X$ be a deterministic process. $X$ is a semimartingale iff $t\mapsto X_t^x$ is c\`adl\`ag and of finite variation on compact intervals for every $x\in\bbr^d$.
\end{proposition}

Every one-dimensional deterministic \emph{Hunt} process is a semimartingale. This is not the case for Markov processes with jumps as the following example illustrates, in which we use the construction principle of the previous section:

\begin{example} \label{ex:cadlagmp}
Let $X$ be the deterministic Feller process with structure $\oplus|\odot|\ominus$. Let the $\odot$-domain be $\{0\}$, that is $X_t^0=0$, and the generating paths
\begin{align*}
  \Phi^{\ominus}&:]-\infty,1[\to]0,\infty[ &\text{ given by } \Phi^{\ominus}(t)&= 1-t\\
  \Phi^{\oplus} &:]-\infty,1[\to]-\infty,0[ &\text{ given by } \Phi^{\oplus}(t)&= t-1. 
\end{align*}
define the process on $]0,\infty[$ respectively on $]-\infty,0[$. Fix a starting point, say $x=1$, and add a jump structure such that the resulting path starting in $x$ jumps so often from $y$ to $-y$ ($y\in]-1,1[$) on the time interval $]0,1]$ that the jumps are not summable. If the only accumulation point of jump times is 1, this results in a c\`adl\`ag Markov process which is not a semimartingale, because the path starting in 1 is of infinite variation on $[0,1]$.
\end{example}

The following result characterizes when a deterministic one-dimensional Markov process is a semimartingale.

\begin{theorem} \label{thm:semimg}
A one-dimensional deterministic c\`adl\`ag Markov process $X$ is a semimartingale iff the jumps of every path $t\mapsto X_t^x$ $(x\in\bbr)$ are locally absolutely summable.
\end{theorem}

This result is wrong for dimensions $d\geq 2$. Just take the function $g:[0,\infty[\to\bbr^2$ given by $g(t)=(f(t),t)'$ where $f(t)$ is a \emph{continuous} function of infinite variation on compacts. The drift in the second variable ensures that the path is injective. Theorem \ref{thm:markovexp} shows that this function is Markov expandable (setting $X_t^x=x$ for every $t\geq 0$ and every $x$ outside the range of $g$). However, since one path is of infinite variation, the process, though it does not have jumps at all, is still no semimartingale. 

In order to prove this theorem we need the following result which belongs to one-dimensional analysis. Even in detailed accounts on the interplay between properties of real valued functions like \cite{muntean}, we have not found it. Therefore, we provide a proof. 

\begin{theorem} \label{thm:analysis}
A function $h:[0,T]\to \bbr$, which is c\`adl\`ag and injective, is of finite variation iff its jumps are absolutely summable.
\end{theorem}

While the `only if'-part of the theorem is well known (injectivity is not even needed for this direction) the `if'-part is more involved. We need the following two lemmas for to prove this direction. The proof of the first one is elementary, while the second one can be found e.g. in \cite{natanson} Corollary VIII.3.1.

\begin{lemma} \label{lem:rearrange}
Let $a<c$ and $f:[a,c[\to\bbr$ be a continuous function such that $f(a)=f(c-)$. Furthermore let $b\in[a,c[$ with $f(a)\neq f(b)$. Let $g:[a,c[\to \bbr$ be a c\`adl\`ag function which is pure jump, that is
\[
g(x)=\sum_{a\leq s\leq x} g(s)-g(s-).
\]
If $\sum_{a\leq s \leq c} \abs{\Delta g(s)} < (1/2)\cdot \abs{f(a)-f(b)}$ then $f+g$ is not injective.
\end{lemma}

\begin{lemma} \label{lem:separate}
Let $a<b<c$. A function $f:[a,c]\to \bbr$ is of finite variation on $[a,c]$ iff it is of finite variation on $[a,b]$ and $[b,c]$.
\end{lemma}

\begin{proof}[Proof of Theorem \ref{thm:analysis}]
Let $h:[0,T]\to \bbr$ be an injective c\`adl\`ag function of bounded variation. Divide this function into two increasing c\`adl\`ag functions:
$h=h^+ - h^-$ such that $h^+$ and $h^-$ do not jump at the same time.
We obtain
\begin{align*}
\sum_{\stackrel{0\leq s \leq T}{\abs{\Delta h(s)}>0}} \abs{\Delta h(s)}&= 
 \sum_{\stackrel{0\leq s \leq T}{\abs{\Delta h^+(s)}>0}} \abs{\Delta h^+(s)}
 + \sum_{\stackrel{0\leq s \leq T}{\abs{\Delta h^-(s)}>0}} \abs{h^-(s)}\\
 &\leq (h^+(T) - h^+(0))+ (h^-(T) - h^-(0)). \\
 &<\infty.
\end{align*}
Now suppose the theorem was false. Then we could find an injective c\`adl\`ag function $h:[0,T]\to\bbr$, with absolutely summable jumps, which is of infinite variation. Divide the c\`adl\`ag function $h$ into its continuous part $f$ and its pure jump part $g$. $g$ is of finite variation, since the jumps of $h$ are absolutely summable on $[0,T]$. Therefore, $f$ has to be of infinite variation on $[0,T]$, because the functions of finite variation form a vector space. Hence, there exists a sequence of partitions $\pi_n:= (0=t_1^n < t_2^n<...<t_{k(n)}^n=T)$ of $[0,T]$ such that
\begin{align} \label{partition}
\sum_{t_j^n,t_{j+1}^n \in \pi_n} \abs{f(t_{j+1}^n) - f(t_j^n)} \xrightarrow[n\to\infty]{}\infty.
\end{align}
Since the jumps of $h$ are absolutely summable, there exists a constant $C_g>0$ such that $\sum_{\abs{\Delta g(s)} > 0} \abs{\Delta g(s)}=C_g <\infty$ and since $f$ is continuous, there exists a constant $C_f>0$ such that $\abs{f}$ is bounded by $C_f$. By \eqref{partition} we obtain that there exists an $N\in\bbn$ such that 
\[
\sum_{t_j^N,t_{j+1}^N \in \pi_N} \abs{f(t_{j+1}^N) - f(t_j^N)} > 4C_g+C_f.
\]
This $N$ is fixed from now on and we write $t_j:=t_j^N$ ($j=1,...,k$) in order to simplify notation. W.l.o.g we assume that $f(0)=0$ and $f\geq 0$. If these properties are not met, we can shift the function (up/down) and analyze the positive and negative part separately. Furthermore we assume w.l.o.g that for every $j$ we have $f(t_{j})-f(t_{j-1})\neq 0$ and
\[
f(t_{j})-f(t_{j-1})<0 \Leftrightarrow f(t_{j+1})-f(t_j)>0,
\]
that is, there are no `useless points' in the partition $\pi_N$. The idea is now to rearrange $f$ (and hence the original function $h$) in such a way that lemma \ref{lem:rearrange} is applicable.\\ 
\emph{step 1: introduce new points into the partition} \\
We set 
\begin{align}
s_1^l&:=t_1 \nonumber \\
s_1^r&:=t_3 \label{twofurther}\\ 
s_2^l&:=\min \{t>s_1^r:f(t)=f(s_1^r) \} \nonumber \\
s_2^r&:=\min \{t_k\in \pi_N: t_k > s_2^l \} \nonumber \\
s_3^l&:=\min \{t>s_2^r:f(t)=f(s_2^r) \} \nonumber \\
... \nonumber
\end{align}
Stop this procedure when reaching the last point $T=t_k$ or the value of $s_j^r$ is zero. In the latter case start the procedure again in $s_{j+1}^l:=s_j^r$ to built up the next `hill'. We call the function $f$ restricted to the set $[s_1^l, s_1^r[\cup ...\cup [s_j^l, s_j^r[ $ a hill though it is not the function which looks like a hill, but the traverse line through the points $(s_1^l,f(s_1^l))$, ..., $(s_1^l,f(s_1^l))$. We continue with the next hill until finally reaching 
\[
T_{new}:=\begin{cases} T & \text{, if} f(t_k)<f(t_{k-1}) \\
                       t_{k-1} &\text{, else.} \end{cases}
\] 
This finishes the first hierarchical level. We have to introduce $T_{new}$, since otherwise, starting from $T$, line \eqref{twofurther} is not defined (we do not need the next, but the next but one point).\\
We left several gaps between the intervals. Now we treat every interval $[s_j^r,s_{j+1}^l]$ in the same way as above, that is, 
\begin{align*}
s_{j1}^l&:=s_1^r\\
s_{j1}^r&:=\min \{t_{k+1}\in \pi_N: t_k > s_{j1}^l \} \\
s_{j2}^l&:=\min \{t>s_{j1}^r:f(t)=f(s_{j1}^r) \} \\
s_{j2}^r&:=\min \{t_k\in \pi_N: t_k > s_{j2}^l \} \\
s_{j3}^l&:=\min \{t>s_{j2}^r:f(t)=f(s_{j2}^r) \} \\
...
\end{align*}
until we reach $s_{j+1}^l$. Of course, it can happen that $s_{j1}^r=s_{j+1}^l$. In this case a third hierarchical level is not needed on the interval $[s_j^r,s_{j+1}^l]$. If there are intervals in the second hierarchical level, where we have to introduce new points $s_{ij}^r$ and $s_i(j+1)^l$ then we treat the interval $[s_{ij}^r, s_i(j+1)^l]$ again as above. We continue with this procedure until the union of all intervals (of all hierarchical levels) covers $[0,T_{new}]$. Note that the number of additional points is finite. It is simple to prove that an upper bound is given by $k^2$, but one can prove better estimates.\\
\emph{step 2: the last hill remains unfinished} \\
Let $s_M^l$ be the maximum of interval-endpoints in the first hierarchical level with the property $f(s_M^l)=0$. If $f(T_{new})\neq 0$ we do not have the property $f(s_M^l)=f(T_{new})$. In order to bring Lemma \ref{lem:rearrange} into account, we introduce a new intermediate point:
\[
f(s_M^m) := \inf\{t>s_M^l: f(t)=f(T_{new}) \}
\]
This gives us two new intervals: one with no intermediate points and the property $f(s_M^l)<f(s_M^m)$ and the second one, $[s_M^m,s_M^r]$ with the property $f(s_M^m)=f(s_M^r)$. The last interval on each hierarchical level might have to be treated in this way. However, the sum of all differences $f(s_M^m)-f(s_M^l)$ plus $f(T)-f(T_{new})$ is bounded by $C_f$.

\scalebox{1} 
{
\begin{pspicture}(0,-2.8629167)(12.665833,2.8629167)
\rput(0.66583335,-2.1370833){\psaxes[linewidth=0.04,arrowsize=0.05291667cm 2.0,arrowlength=1.4,arrowinset=0.4,labels=none,ticksize=0.10583333cm]{->}(0,0)(0,0)(12,5)}
\psline[linewidth=0.04cm](0.6458333,-2.1570833)(1.6458334,0.86291665)
\psline[linewidth=0.04cm](1.6658334,0.86291665)(2.6058333,-0.09708334)
\psline[linewidth=0.04cm](2.6058333,-0.077083334)(3.6258333,1.8829167)
\psline[linewidth=0.04cm](3.6458333,1.8829167)(6.6858335,-2.1170833)
\psline[linewidth=0.04cm](6.6858335,-2.0970833)(9.645833,-0.11708333)
\psline[linewidth=0.04cm](9.6258335,-0.11708333)(10.605833,-1.1170833)
\psline[linewidth=0.04cm](10.6258335,-1.1170833)(11.665833,0.86291665)
\pscustom[linewidth=0.04]
{
\newpath
\moveto(3.6058333,1.8629167)
\lineto(3.6558332,1.7429167)
\curveto(3.6808333,1.6829166)(3.7158334,1.5979167)(3.7258334,1.5729166)
\curveto(3.7358334,1.5479167)(3.7558334,1.5029167)(3.7658334,1.4829167)
\curveto(3.7758334,1.4629166)(3.7908332,1.4179167)(3.7958333,1.3929167)
\curveto(3.8008332,1.3679167)(3.8208334,1.3279166)(3.8358333,1.3129166)
\curveto(3.8508334,1.2979167)(3.8708334,1.2629167)(3.8758333,1.2429167)
\curveto(3.8808334,1.2229167)(3.8858333,1.1729167)(3.8858333,1.1429167)
\curveto(3.8858333,1.1129167)(3.8958333,1.0629166)(3.9058332,1.0429167)
\curveto(3.9158332,1.0229167)(3.9308333,0.97791666)(3.9358332,0.9529167)
\curveto(3.9408333,0.92791665)(3.9508333,0.8779167)(3.9558334,0.85291666)
\curveto(3.9608333,0.8279167)(3.9708333,0.78291667)(3.9758334,0.7629167)
\curveto(3.9808333,0.74291664)(3.9958334,0.7129167)(4.005833,0.7029167)
\curveto(4.0158334,0.6929167)(4.0408335,0.6929167)(4.0558333,0.7029167)
\curveto(4.070833,0.7129167)(4.090833,0.74791664)(4.0958333,0.7729167)
\curveto(4.1008334,0.79791665)(4.130833,0.81291664)(4.1558332,0.80291665)
\curveto(4.1808333,0.79291666)(4.2058334,0.7529167)(4.2058334,0.72291666)
\curveto(4.2058334,0.6929167)(4.2058334,0.6379167)(4.2058334,0.61291665)
\curveto(4.2058334,0.5879167)(4.2058334,0.53791666)(4.2058334,0.5129167)
\curveto(4.2058334,0.48791668)(4.215833,0.47791666)(4.2258334,0.49291667)
\curveto(4.235833,0.5079167)(4.255833,0.53291667)(4.2658334,0.54291666)
\curveto(4.275833,0.55291665)(4.295833,0.53291667)(4.3058333,0.5029167)
\curveto(4.315833,0.47291666)(4.3308334,0.41791666)(4.3358335,0.39291668)
\curveto(4.340833,0.36791667)(4.3558335,0.32791665)(4.3658333,0.31291667)
\curveto(4.3758335,0.29791665)(4.400833,0.25791666)(4.4158335,0.23291667)
\curveto(4.4308333,0.20791666)(4.4558334,0.15791667)(4.465833,0.13291666)
\curveto(4.4758334,0.10791667)(4.4958334,0.06791666)(4.505833,0.052916665)
\curveto(4.5158334,0.037916664)(4.5408335,0.002916665)(4.5558333,-0.017083336)
\curveto(4.570833,-0.037083335)(4.5958333,-0.07208333)(4.6058335,-0.08708333)
\curveto(4.6158333,-0.10208333)(4.6408334,-0.13208334)(4.6558332,-0.14708334)
\curveto(4.670833,-0.16208333)(4.7308335,-0.17208333)(4.775833,-0.16708334)
\curveto(4.820833,-0.16208333)(4.8958335,-0.16208333)(4.925833,-0.16708334)
\curveto(4.9558334,-0.17208333)(4.9958334,-0.20208333)(5.005833,-0.22708334)
\curveto(5.0158334,-0.25208333)(5.0358334,-0.29708335)(5.045833,-0.31708333)
\curveto(5.0558333,-0.33708334)(5.0858335,-0.36708334)(5.1058335,-0.37708333)
\curveto(5.1258335,-0.38708332)(5.170833,-0.40208334)(5.195833,-0.40708333)
\curveto(5.2208333,-0.41208333)(5.2658334,-0.43708333)(5.2858334,-0.45708334)
\curveto(5.3058333,-0.47708333)(5.3358335,-0.51708335)(5.3458333,-0.5370833)
\curveto(5.3558335,-0.5570833)(5.380833,-0.59708333)(5.3958335,-0.6170833)
\curveto(5.4108334,-0.63708335)(5.4358335,-0.6720833)(5.445833,-0.68708336)
\curveto(5.4558334,-0.70208335)(5.570833,-0.6620833)(5.675833,-0.6070833)
\curveto(5.7808332,-0.5520833)(5.8958335,-0.5420833)(5.9058332,-0.58708334)
\curveto(5.9158335,-0.63208336)(5.9308333,-0.69708335)(5.9358335,-0.71708333)
\curveto(5.940833,-0.7370833)(5.9608335,-0.77708334)(5.9758334,-0.7970833)
\curveto(5.9908333,-0.81708336)(6.025833,-0.84708333)(6.045833,-0.8570833)
\curveto(6.065833,-0.8670833)(6.1208334,-0.88708335)(6.1558332,-0.89708334)
\curveto(6.190833,-0.90708333)(6.235833,-0.94208336)(6.2458334,-0.96708333)
\curveto(6.255833,-0.9920833)(6.2708335,-1.0420834)(6.275833,-1.0670834)
\curveto(6.2808332,-1.0920833)(6.2858334,-1.1420833)(6.2858334,-1.1670834)
\curveto(6.2858334,-1.1920834)(6.3258333,-1.2270833)(6.3658333,-1.2370833)
\curveto(6.4058332,-1.2470833)(6.4508333,-1.2870834)(6.4558334,-1.3170834)
\curveto(6.4608335,-1.3470833)(6.465833,-1.4120834)(6.465833,-1.4470834)
\curveto(6.465833,-1.4820833)(6.465833,-1.5420834)(6.465833,-1.5670834)
\curveto(6.465833,-1.5920833)(6.4808335,-1.6370833)(6.4958334,-1.6570834)
\curveto(6.5108333,-1.6770834)(6.5408335,-1.7070833)(6.5558333,-1.7170833)
\curveto(6.570833,-1.7270833)(6.6008334,-1.7620833)(6.6158333,-1.7870834)
\curveto(6.630833,-1.8120834)(6.6558332,-1.8520833)(6.6658335,-1.8670833)
\curveto(6.675833,-1.8820833)(6.6858335,-1.9220834)(6.6858335,-1.9470834)
\curveto(6.6858335,-1.9720833)(6.690833,-2.0220833)(6.695833,-2.0470834)
\curveto(6.7008333,-2.0720832)(6.7208333,-2.1120834)(6.735833,-2.1270833)
\curveto(6.7508335,-2.1420834)(6.7808332,-2.1670833)(6.8258333,-2.1970832)
}
\psdots[dotsize=0.12](4.6258335,-0.13708334)
\psdots[dotsize=0.12](7.6658335,-1.1370833)
\pscustom[linewidth=0.04]
{
\newpath
\moveto(6.8258333,-2.1970832)
\lineto(6.8658333,-2.2770834)
\curveto(6.8858333,-2.3170834)(6.925833,-2.3620834)(6.945833,-2.3670833)
\curveto(6.965833,-2.3720834)(6.9958334,-2.3620834)(7.005833,-2.3470833)
\curveto(7.0158334,-2.3320832)(7.0308332,-2.2770834)(7.0358334,-2.2370834)
\curveto(7.0408335,-2.1970832)(7.045833,-2.1270833)(7.045833,-2.0970833)
\curveto(7.045833,-2.0670834)(7.0558333,-2.0070834)(7.065833,-1.9770833)
\curveto(7.0758333,-1.9470834)(7.090833,-1.9020833)(7.0958333,-1.8870833)
\curveto(7.1008334,-1.8720833)(7.1258335,-1.8720833)(7.1458335,-1.8870833)
\curveto(7.1658335,-1.9020833)(7.195833,-1.8970833)(7.2058334,-1.8770833)
\curveto(7.215833,-1.8570833)(7.2608333,-1.8070834)(7.295833,-1.7770833)
\curveto(7.3308334,-1.7470833)(7.3708334,-1.6870834)(7.3758335,-1.6570834)
\curveto(7.380833,-1.6270833)(7.3958335,-1.5670834)(7.4058332,-1.5370834)
\curveto(7.4158335,-1.5070833)(7.4308333,-1.4570833)(7.4358335,-1.4370834)
\curveto(7.440833,-1.4170834)(7.4608335,-1.3720833)(7.4758334,-1.3470833)
\curveto(7.4908333,-1.3220834)(7.5208335,-1.2970834)(7.565833,-1.2970834)
}
\pscustom[linewidth=0.04]
{
\newpath
\moveto(7.565833,-1.2770833)
\lineto(7.6058335,-1.2670833)
\curveto(7.6258335,-1.2620833)(7.6558332,-1.2270833)(7.6858335,-1.1370833)
}
\pscustom[linewidth=0.04]
{
\newpath
\moveto(7.6858335,-1.1370833)
\lineto(7.715833,-1.0870833)
\curveto(7.7308335,-1.0620834)(7.7608333,-1.0120833)(7.775833,-0.9870833)
\curveto(7.7908335,-0.96208334)(7.8308334,-0.9270833)(7.8558335,-0.9170833)
\curveto(7.880833,-0.90708333)(7.940833,-0.8720833)(7.9758334,-0.84708333)
\curveto(8.010834,-0.82208335)(8.115833,-0.8020833)(8.185833,-0.8070833)
\curveto(8.255834,-0.81208336)(8.340834,-0.83708334)(8.355833,-0.8570833)
\curveto(8.370833,-0.87708336)(8.405833,-0.90208334)(8.425834,-0.90708333)
\curveto(8.445833,-0.9120833)(8.485833,-0.89208335)(8.505834,-0.8670833)
\curveto(8.525833,-0.84208333)(8.550834,-0.7970833)(8.555833,-0.77708334)
\curveto(8.560833,-0.75708336)(8.580833,-0.71208334)(8.595834,-0.68708336)
\curveto(8.610833,-0.6620833)(8.655833,-0.6120833)(8.685833,-0.58708334)
\curveto(8.715834,-0.56208336)(8.780833,-0.5520833)(8.815833,-0.56708336)
\curveto(8.850833,-0.58208334)(8.915833,-0.6220833)(8.945833,-0.64708334)
\curveto(8.975833,-0.6720833)(9.025833,-0.71208334)(9.045834,-0.7270833)
\curveto(9.065833,-0.7420833)(9.105833,-0.77208334)(9.1258335,-0.7870833)
\curveto(9.145833,-0.8020833)(9.200833,-0.7920833)(9.235833,-0.76708335)
\curveto(9.270833,-0.7420833)(9.320833,-0.68708336)(9.335834,-0.65708333)
\curveto(9.350833,-0.62708336)(9.395833,-0.56708336)(9.425834,-0.5370833)
\curveto(9.455833,-0.50708336)(9.510834,-0.44208333)(9.535833,-0.40708333)
\curveto(9.560833,-0.37208334)(9.595834,-0.31708333)(9.605833,-0.29708335)
\curveto(9.615833,-0.27708334)(9.630834,-0.23208334)(9.635834,-0.20708333)
\curveto(9.640833,-0.18208334)(9.645833,-0.14208333)(9.645833,-0.09708334)
}
\rput(0.5,-2.6){$0=t_1$} \rput(0.85,-3.32){$s_1^l$}
\rput(1.75,-2.6){$t_2$}
\rput(2.75,-2.6){$t_3$} \rput(2.75,-3.32){$s_1^r$} \rput(2.8,-4){$s_{11}^l$}
\rput(3.75,-2.6){$t_4$} 
  \rput(4.75,-3.32){$s_2^l$} \rput(4.8,-4){$s_{11}^r$}
\rput(6.75,-2.6){$t_5$} \rput(6.65,-3.1){$\overbrace{s_2^r=s_3^l}^{}$}
  \rput(7.75,-3.32){$s_3^m$}
\rput(9.75,-2.6){$t_6$} \rput(9.75,-3.32){$s_3^r$}
\rput(10.75,-2.6){$t_7$} \rput(11,-3.32){$T_{new}$}
\rput(12.18,-2.6){$t_8=T$}
\rput(14.18,-3.32){$1^{st}$ level} \rput(14.18,-4){$2^{nd}$ level}
\end{pspicture} 
}

\vspace{2cm}
\emph{step 3: rearranging} \\
Now we shift the (right-open) intervals as well as the functions $f,g$ and $h$: we write $s_j$ for the length of the interval $[s_j^l,s_j^r[$ and set
\begin{align*}
f(t) & \text{ on } [0,s_1[ \\
f(t-s_1+s_2^l) & \text{ on } [s_1,s_1+s_2[ \\
... & \\
f\left(t-\left(\sum_{j=1}^{k-1} s_j \right)+ s_k^l \right) & \text{ on } \left[ \sum_{j=1}^{k-1} s_j , \sum_{j=1}^k s_j \right[.
\end{align*}
on the right we continue with the intervals of the second hierarchical level, then of the third and so on. Finally we plug together the left intervals defined in step 2 and put them to the very right. Just the interval $[T_{new},T]$ remains unchanged. This results in a new function $\widetilde{f}$. $\widetilde{g}$ and $\widetilde{h}$ are defined analogously. In particular we still have $\widetilde{f}+\widetilde{g}=\widetilde{h}$. $\widetilde{g}$ is still pure jump, but $\widetilde{f}$ is not continuous. \\
\emph{step 4: not all hills can be made injective by adding jumps} \\
The function $\widetilde{f}$ has the following property: it can be written as ($p\in\bbn$)
\[
  \widetilde{f}=\left(\sum_{j=1}^{p} f^{(j)} \cdot 1_{[a^{(2j)},a^{(2j+2)}[}\right) + f^{(2p+2)}\cdot 1_{[a^{(2p+2)},T[}
\]
where $a^{(j+1)}>a^{(j)}$, $\widetilde{f}(a^{(j)})=\widetilde{f}(a^{(j+2)}-)$
and 
\[
  \sum_{j=1}^p \abs{\widetilde{f}(a^{(2j+2)}-)-\widetilde{f}(a^{(2j+1)})}+ 
               \abs{\widetilde{f}(a^{(2j+1)})-\widetilde{f}(a^{(2j)})} >4C_g
\]
($C_f$ does no longer appear, because of the last unfinished hill). We call every part of the function $f^{(j)} \cdot 1_{[a^{(2j)},a^{(2j+2)}[}$ a hill. $\widetilde{f}$ only jumps in points $a^{(2j)}$ $(j=1,...,p+1)$, i.e., between the hills. No matter how we divide the jumps $g$ to the intervals 
$[a^{(2j)},a^{(2j+2)}[$, we always get at least one interval where the requirements of Lemma \ref{lem:rearrange} are met. On this interval the function $\widetilde{h}=\widetilde{f}+\widetilde{g}$ in not injective. Since we have only rearranged $h$ in order to obtain $\widetilde{h}$, the original function $h$ in not injective, as well. Hence we obtain a contradiction. The function $h$ can not be of infinite variation and the theorem is proved.
\end{proof}


The proof of our main result is now relatively simple:

\begin{proof}[Proof of theorem \ref{thm:semimg}]
The if part is clear by the above Theorem. Now let the jumps of $X$ be locally absolutely summable. Fix $x\in\bbr$ and $T>0$. By Theorem \ref{thm:markovexp} we know what the paths of a Markov process look like. We exclude the trivial case $X_t^x=x$. By Lemma \ref{lem:separate} we can and will assume that the path $h(t):=X^x_t$ is injective on $[0,T]$ for a $T>0$. If this was not the case it would be sufficient to analyze the path up to time $t_1\leq T$ (cf. Remark \ref{def:tzero}) and then separate the interval into 
\[
[0,t_1-d], \ [t_1-d,t_1], [t_1,t_1+d], [t_1+d,t_1+2d],...,[t_1+nd,T]
\]
where $d:=(t_1-t_0)/2$. We divide by 2, because $h(t_1)=h(t_0)$. The result follows from Theorem \ref{thm:analysis}.
\end{proof}






\section{The Time Inhomogeneous Case}


For most of our results time homogeneity is a key ingredient and the results (or analogous of the results) do not hold for the time inhomogeneous case. However, complementing the results of Section 2 we have the following:

\begin{proposition}
Every function $f:[0,\infty[ \to \bbr^d$ is expandable to a space homogeneous Markov process (which is in general not time homogeneous).
\end{proposition}

\begin{proof}
Just set $X_t^x = x+(f(t)-f(0))$.
\end{proof}

The proposition says in particular, that \emph{every} function is expandable to a time inhomogeneous Markov process. The requirements on $f$ (cf. Theorem \ref{thm:markovexp}) are not needed. 

To find a criterion for a time inhomogeneous Markov process to be a semimartingale (cf. Section 4), which is better than Proposition \ref{prop:detsemimg} is hopeless. Just take $\Phi$ to be any function of unbounded variation plus the standard construction principle. Note that we can not use the space-time Markov process $(X_t,t)$ in order to establish an analogous result, since our result does only hold in one dimension. This is interesting from a philosophical point-of-view: in the literature one sometimes gets the impression that time inhomogeneous Markov processes are not of any interest since one can \emph{always} use the space-time Markov process in order to transfer results. Here we can see that this is not the case. Compare in this context example \ref{ex:space-time}.

\begin{remark}
The difference between the time homogeneous and the time inhomogeneous case can be emphasized by having a look at the following situation: two continuous paths $X^x$ and $X^y$ with $x<y$ hit each other for the first time at $t_0$. By the Markov property they have to `stick together' afterwards. 
In the \emph{time homogeneous} case this is only possible if $X^x$ is monotone increasing on $[0,t_0[$ and $X^y$ is monotone decreasing on $[0,t_0[$. Furthermore $X^x$ and $x^y$ are constant on $[t_0,\infty[$. Every path $X^z$ with ($x\leq z \leq y$) is known and if for any $w\in ]-\infty,x[\cup]y,\infty[$ there exists a $t_1 >0$ such that $X_{t_1}^w\in[x,y]$ we know how the path behaves on $[t_1,\infty[$. In the \emph{time inhomogeneous} case we do not know anything about the behavior of $X^x$ (resp. $X^y$) on $]t_0,\infty[$ or about any path $X_t^w$ with $w\in]-\infty, x[\cup ]y,\infty[$. The only thing we can say is that for every continuous path $X^z$ ($x\leq z \leq y$) we have $X^z=X^y$ on $[t_0,\infty[$.
\begin{center}
\includegraphics[width=5cm, angle=270]{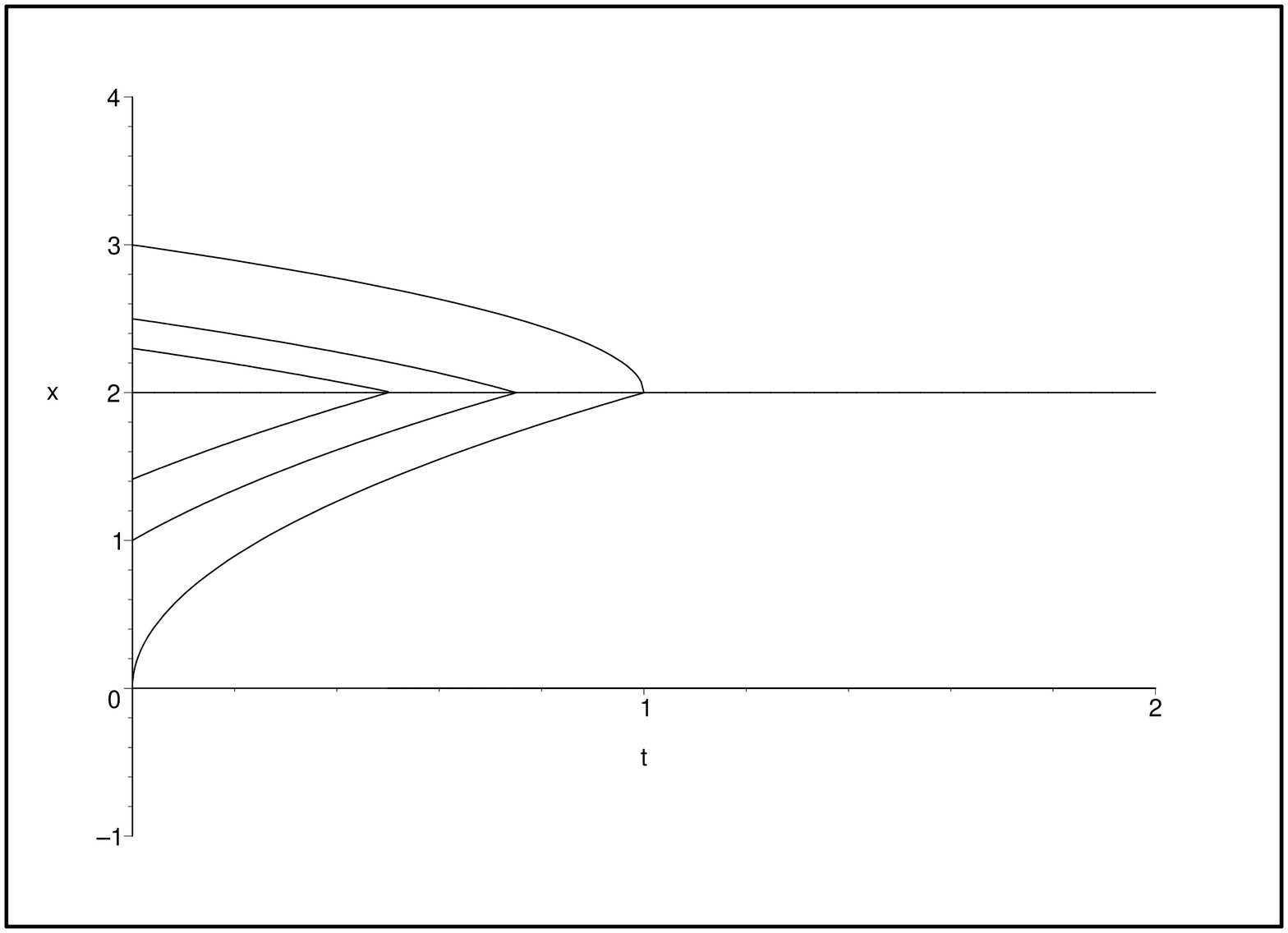}
\includegraphics[width=5cm, angle=270]{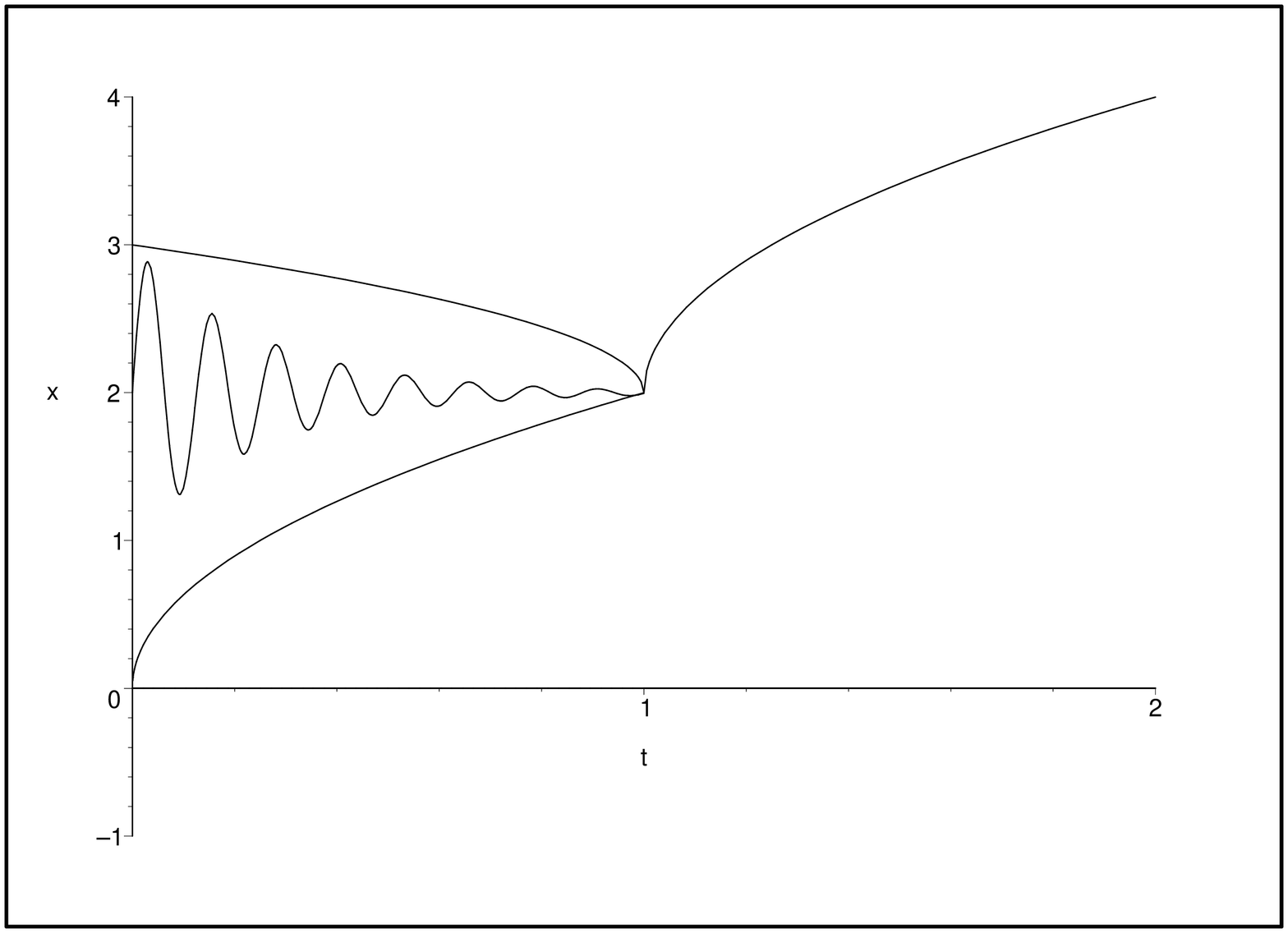}
\end{center}
\end{remark}

\hspace*{26mm} time homogeneous \hspace{33mm} time inhomogeneous

\section{Further Examples}

In this final section we have collected several examples which complement the results of the previous sections.

\begin{example} \label{ex:cantorprocess}
Let $h:\bbr\to [0,1]$ be the Cantor function (cf. \cite{cantorfunction} and \cite{elstrodt} Section 8.4) and define $g:[0,1]\to [0,1]$ by $g(y):= (1/2)(h(y)+y)$. For $x\in\bbr$ let $\Phi(x):=g(x-[x])+[x]$ with $g$ defined above and $x\mapsto[x]$ denoting the floor function. The stochastic process $X=(X_t)_{t\geq 0}$ defined by
\[
  X_t^x:=\Phi(t+\Phi^{-1}(x))
\]
is called \emph{Cantor process}. In \cite{detmp1} we have seen that this process is Feller, but not rich and that it is not an It\^o process, but a Hunt semimartingale.
\end{example}

\begin{example}\label{ex:bjoern}
Define $X^x_t=x\exp(t)$. The structure of this process is $\ominus|\odot|\oplus$, with the generating paths $t\mapsto \exp(t)$ respective $t\mapsto -\exp(t)$ on $]0,\infty[$ respective $]-\infty,0[$. Though the generating paths are diverging the process is still Feller.
\end{example}

\begin{example}\label{ex:killing}
A deterministic L\'evy process does not admit a killing and the only way how a killing can arise in the context of deterministic Hunt processes is by a state-space dependent drift which reaches infinity in finite time. If jumps are taken into account, we have two new ways of killing a process: first, we could start e.g. with a L\'evy process $X_t^x=x+ct$ ($c\neq 0$) and add a jump structure like
\[
  (a_n,b_n) =  \left(c\cdot \frac{2^n-1}{2^n}+\sum_{j=0}^{n-2} 2^j , c\cdot \frac{2^n-1}{2^n}+\sum_{j=0}^{n-1} 2^j \right).
\]
This results in a deterministic Hunt process which reaches infinity in finite time, if we start in
\[
  \bigcup_{n\in \bbn_0} \left[ 2^n + c\cdot \frac{2^n-1}{2^n}, 2^n + c\cdot \frac{2^{n+1}-1}{2^{n+1}} \right[.
\]
The second way is by adding an `instant killing', i.e. if the process reaches a certain set, it is sent directly to a cemetery state $\Delta$. \\
Deterministic processes which are strictly monotone increasing can be used as a time change in order to kill a given arbitrary process. 
\end{example}

\begin{example}\label{ex:notmarkov}
Let $X_t^x=x+(t\cdot (-x))=x\cdot (1-t)$. All paths are zero at time one, but the behavior afterward depends on the starting point $x$. This is a process which is not Markovian, but still a semimartingale and even a homogeneous diffusion in the sense of Definition III.2.18 of \cite{jacodshir} with respect to every starting point. Unlike Example \ref{ex:hdwjunif} the differential characteristic $\ell$ depends on the starting point here.
\end{example}



\begin{example} \label{ex:sum}
We show that the sum of two simple time homogeneous Markov processes starting in zero does not have to be again a time homogeneous Markov process (no matter how the filtrations are chosen): let $X_t:=t\cdot 1_{[0,1[}(t)+1_{[1,\infty[}(t)$ and $Y_t:=-t\cdot 1_{[0,2[}(t)-2\cdot 1_{[2,\infty[}(t)$. Their sum is 
\[
Z_t=X_t+Y_t=(1-t)\cdot 1_{[1,2[}(t)-1_{[2,\infty[}(t)
\]
contradicting time homogeneity. 
\end{example}

\begin{example}\label{ex:loopjumps}
The following path reaches zero (in the sense of left-limits) infinitely often, but in two different ways, which can not be described by means of monotonicity. Let
\[
X_t^1 := \sum_{n\in (2\bbn_0)} ((n+1)-t) 1_{A+n}(t) + (t-(n+1)) 1_{B+n}(t) + (t-(n+2)) 1_{A+(n+1)}(t) +((n+2)-t) 1_{B+(n+1)}(t) 
\]
with $A$ and $B$ as in Example \ref{ex:AtoD}:
\begin{center}
\setlength{\unitlength}{2cm}
\begin{picture}(4,2)
\put(0,0){\vector(0,1){2}}
\put(0,1){\vector(1,0){4.1}}
\linethickness{0.5mm}
\put(0,2){\line(1,-1){0.5}}
\put(0.5,0.5){\line(1,1){0.25}}
\put(0.75,1.25){\line(1,-1){0.125}}
\put(0.875,0.875){\line(1,1){0.0625}}
\put(0.9375,1.0625){\line(1,-1){0.03125}}
\put(0.96875,0.96875){\line(1,1){0.015625}}
\put(1,0.8){1}
\put(2,0.8){2}
\put(3,0.8){3}
\put(4,0.8){4}
\put(1,0){\line(1,1){0.5}}
\put(1.5,1.5){\line(1,-1){0.25}}
\put(1.75,0.75){\line(1,1){0.125}}
\put(1.875,1.125){\line(1,-1){0.0625}}
\put(1.9375,0.9375){\line(1,1){0.03125}}
\put(1.96875,1.03125){\line(1,-1){0.015625}}
\put(2,2){\line(1,-1){0.5}}
\put(2.5,0.5){\line(1,1){0.25}}
\put(2.75,1.25){\line(1,-1){0.125}}
\put(2.875,0.875){\line(1,1){0.0625}}
\put(2.9375,1.0625){\line(1,-1){0.03125}}
\put(2.96875,0.96875){\line(1,1){0.015625}}
\put(3,0){\line(1,1){0.5}}
\put(3.5,1.5){\line(1,-1){0.25}}
\put(3.75,0.75){\line(1,1){0.125}}
\put(3.875,1.125){\line(1,-1){0.0625}}
\put(3.9375,0.9375){\line(1,1){0.03125}}
\put(3.96875,1.03125){\line(1,-1){0.015625}}
\end{picture}
\end{center}
\end{example}

\begin{example}\label{ex:infinitelandingpoints}
For $n\in\bbn$ let
\[
X^{-n}_t := (t-n)1_{\left[ 0, \frac{2^n-1}{2^n}\right[} (t) + \sum_{j=n}^\infty \left( \frac{1}{2^j} t + \frac{2^j-1}{2^j} - \frac{n}{2^{2j+1}} \right) 1_{\left[\frac{2^j-1}{2^j} , \frac{2^{j+1}-1}{2^{j+1}} \right[} (t) +n 1_{[1,\infty[}(t)
\]
this results in infinitely many paths with $\lim_{t\uparrow 1} X_t^{-n}=0$ but $X_1^{-n}=n$.
\end{example}

\begin{example}\label{ex:space-time}
Here we consider the space-time Markov process. For this purpose let $X$ be a one-dimensional deterministic Markov process which is not homogeneous in time, i.e. there exist $s_1<t_1$ and $s_2<t_2$ such that $t_1-s_1=t_2-s_2$ and
\[
X_{s_1}^x=X_{s_2}^y \text{ but } X_{t_1}^x \neq X_{t_2}^y
\]
for starting points $x$ and $y$. It is a well-known fact that the corresponding space-time process $(X_t,t)'$ is homogeneous in time. This is due to the simple fact that the new paths do not coincide in $s_1$ and $s_2$. This does \emph{not} mean that the structure of the process becomes simpler. In the context of universal Markov processes it means that we get a whole new dimension of possible starting points. For these new starting points we get the paths by using the injective functions $\Phi^x(t):=(X_t^x,t)$ ($x\in\bbr$) as generating paths. This results in a process $Y$ on $\bbr\times [0,\infty[$ and makes things more complicated e.g. in the context of generators: assume that the paths of $X$ are continuously differentiable and write $f^x(t):=X_t^x$. The symbol of the generator of $Y$ is 
\[
p(y,\xi) = -i\xi_1 \partial^+ f((\Phi^x)^{-1} (y)) - i\xi_2
\]
where $x$ is a starting point in the original state space for which we have $y=X_{y_2}^x$. This might be the case for different $x$. However, for these different starting points the right-hand side derivatives in $(\Phi^x)^{-1} (y)$ do coincide by the Markov property. Therefore the above expression is well defined. 
If we considered the original process instead, we would have a family of generators with the time dependent symbol
\[
p_t(x,\eta) = -i\eta \partial^+ f(t).
\]
which appears more natural to us. 
This shows that it is sometimes worthwhile to consider the original process and that time inhomogeneous Markov processes have some value on their own right.
\end{example}

\section*{Acknowledgements}

Part of this work has been done while the author was visiting K. Bogdan (Wroclaw University of Technology) and Y. Xiao (Michigan State University). He would like to thank both professors and their institutes for the kind hospitality and for interesting discussions.  Financial support by the German Science Foundation (DFG) for the project SCHN1231/1-1 is gratefully acknowledged.


\end{document}